\documentclass[12pt]{article}
\usepackage[utf8]{inputenc}
\usepackage{amsmath}
\usepackage{amsthm}
\usepackage{amssymb}
\usepackage{cite}
\usepackage{cleveref}
\usepackage{multicol,url}
\usepackage{tikz-cd}
\usepackage{enumitem}
\usepackage{mathrsfs}
\usepackage{extarrows}
\usepackage{amsfonts}
\usepackage{mathtools}

\newtheorem{theorem}{Theorem}[section]
\newtheorem{corollary}[theorem]{Corollary}
\newtheorem{lemma}[theorem]{Lemma}
\newtheorem{proposition}[theorem]{Proposition}
\newtheorem{definition}[theorem]{Definition}
\newtheorem{problem}{Problem}
\newtheorem*{claim}{Claim}
\newtheorem*{remark}{Remark}
\newtheorem*{note}{Note}

\renewcommand{\and}{\text{ and }}

\newcommand{\BC}{\begin{center}}
\newcommand{\EC}{\end{center}}
\newcommand{\BCS}{\begin{cases}}
\newcommand{\ECS}{\end{cases}}
\newcommand{\BM}{\begin{pmatrix}}
\newcommand{\EM}{\end{pmatrix}}
\newcommand{\BD}{\begin{tikzcd}}
\newcommand{\ED}{\end{tikzcd}}
\newcommand{\BE}{\begin{equation}}
\newcommand{\EE}{\end{equation}}
\newcommand{\thm}{\begin{theorem}}
\newcommand{\ethm}{\end{theorem}}
\newcommand{\prf}{\begin{proof}}
\newcommand{\eprf}{\end{proof}}
\newcommand{\cor}{\begin{corollary}}
\newcommand{\ecor}{\end{corollary}}
\newcommand{\prop}{\begin{proposition}}
\newcommand{\eprop}{\end{proposition}}
\newcommand{\prob}[1]{\begin{problem}[#1]}
\newcommand{\eprob}{\end{problem}}
\newcommand{\lem}{\begin{lemma}} 
\newcommand{\elem}{\end{lemma}}
\newcommand{\defi}{\begin{definition}}
\newcommand{\edefi}{\end{definition}}
\newcommand{\clm}{\begin{claim}}
\newcommand{\eclm}{\end{claim}}
\newcommand{\rem}{\begin{remark}}
\newcommand{\erem}{\end{remark}}
\newcommand{\nt}{\begin{note}}
\newcommand{\ent}{\end{note}}
\newcommand{\renumerate}{\begin{enumerate}[label=(\roman*)]}
\newcommand{\eenumerate}{\end{enumerate}}
\newcommand{\lenumerate}{\begin{enumerate}[label=(\alph*)]}
\newcommand{\nenumerate}{\begin{enumerate}[label=(\arabic*)]}

\begin{document}

\title{\bf
Using the quantum torus to investigate the $q$-Onsager algebra
}
\author{
Owen Goff\textsuperscript{a}}
\date{}
\maketitle

\begin{abstract}

The $q$-Onsager algebra, denoted by $O_q$, is defined by generators $W_0, W_1$ and two relations called the $q$-Dolan-Grady relations. {In 2017, Baseilhac and Kolb gave some elements of $O_q$ that form a Poincar\'e-Birkhoff-Witt basis.} The quantum torus, denoted by $T_q$, is defined by generators $x, y, x^{-1}, y^{-1}$ and relations $$xx^{-1} = 1 = x^{-1}x, \qquad yy^{-1} = 1 = y^{-1}y, \qquad xy=q^2yx.$$ The set $\{x^iy^j | i,j \in \mathbb{Z} \}$ is a basis for $T_q$. It is known that there is an algebra homomorphism $p: O_q \mapsto T_q$ that sends $W_0 \mapsto x+x^{-1}$ and $W_1 \mapsto y+y^{-1}.$ In 2020, Lu and Wang displayed a variation of $O_q$, denoted by $\tilde{\mathbf{U}}^{\imath}$. Lu and Wang gave a surjective algebra homomorphism $\upsilon : \tilde{\mathbf{U}}^{\imath} \mapsto O_q.$

\medskip 
In their consideration of $\tilde{\mathbf{U}}^{\imath}$, Lu and Wang introduced some elements \begin{equation} \label{intrp503}
 \{B_{1,r}\}_{r \in \mathbb{Z}}, \qquad \{H'_n\}_{n=1}^{\infty}, \qquad \{H_n\}_{n=1}^{\infty}, \qquad \{\Theta'_n\}_{n=1}^{\infty}, \qquad \{\Theta_n\}_{n=1}^{\infty}. \nonumber
\end{equation}

These elements are defined using recursive formulas and generating functions, and it is difficult to express them in closed form. A similar problem applies to the Baseilhac-Kolb elements of $O_q$. To mitigate this difficulty, we map everything to $T_q$ using $p$ and $\upsilon$. In our main results, we express the resulting images in the basis for $T_q$ and also in an attractive closed form.
 
 \end{abstract}
\bigskip
\section{Introduction}

  The $q$-Onsager algebra, which we denote by $O_q$, first appeared in the topic of algebraic graph theory \cite[Lemma 5.4]{assocscheme}. Applications of $O_q$ in algebraic combinatorics appear in many papers, including \cite{Owen, xxy, IAT, 1865045}. The algebra $O_q$ has recently been applied to statistical mechanics \cite{BB, in, Source16, ddg, XX} 
 and quantum groups \cite[Section 2.2]{p}.  The algebra $O_q$ is defined by generators $W_0, W_1$ and two relations called the $q$-Dolan-Grady relations; see Definition \ref{oq}. In \cite{p}, Baseilhac and Kolb gave a Poincar\'{e}-Birkhoff-Witt (PBW) basis for $O_q$; the elements of this PBW basis are now called the {Baseilhac-Kolb elements} of $O_q$.

  \medskip 
  
Next we recall the quantum torus, which we denote by $T_q$. The algebra $T_q$ is defined by the generators $x$, $y$, $x^{-1}$, $y^{-1}$ and the relations $$xx^{-1} = 1 = x^{-1}x, \qquad yy^{-1} = 1 = y^{-1}y, \qquad xy=q^2yx.$$ By \cite[p. 3]{Magic}, the elements $\{x^iy^j\vert i, j \in \mathbb{Z}\}$ form a basis for $T_q$. In \cite{Owen}, we displayed an algebra homomorphism $p: O_q \rightarrow T_q$ that sends $W_0 \mapsto x + x^{-1}$ and $W_1 \mapsto y + y^{-1}$.

\medskip

In \cite{Drinfeld}, Lu and Wang introduced a variation of $O_q$, denoted by $\tilde{\mathbf{U}}^{\imath}$. In \cite[Remark 2.2]{Drinfeld}, Lu and Wang display a surjective algebra homomorphism $\upsilon:\tilde{\mathbf{U}}^{\imath} \mapsto O_q$ (see also \cite[Remark 5.7]{T2} and\cite[Remark 3.2]{T1}). We consider the composition \begin{equation} \label{vp}
\tilde{\mathbf{U}}^{\imath} \xlongrightarrow{\upsilon} O_q \xlongrightarrow{p} T_q.
\end{equation}

In \cite{Drinfeld}, Lu and Wang introduced some elements of $\tilde{\mathbf{U}}^{\imath}$, analogous to the Baseilhac-Kolb elements of $O_q$ but with nicer algebraic properties. These elements are the real $q$-root vectors $\{B_{1,r}\}_{r\in \mathbb{Z}}$ and the two types of imaginary $q$-root vectors, $\{H'_n\}_{n=1}^{\infty}$ and $\{H_n\}_{n=1}^{\infty}$. The elements $\{B_{1,r}\}_{r \in \mathbb{Z}}$ and $\{H_n\}_{n=1}^{\infty}$ form a PBW basis analogous to the Beck PBW basis for the positive part of $U_q(\widehat{\mathfrak{sl}_2})$ described in \cite{Beck}, while the elements $\{H'_n\}_{n=1}^{\infty}$ have connections to $\imath$Hall algebras (See \cite[Section 2.5]{Drinfeld}). To construct the imaginary $q$-root vectors, Lu and Wang introduced some elements $\{\Theta'_n\}_{n=1}^{\infty}$ and $\{\Theta_n\}_{n=1}^{\infty}$ of $\tilde{\mathbf{U}}^{\imath}$.

\medskip

In the preceding paragraph, we mentioned the elements
\begin{equation} \label{boom}
\{B_{1,r}\}_{r \in \mathbb{Z}}, \qquad \{H'_n\}_{n=1}^{\infty}, \qquad \{H_n\}_{n=1}^{\infty}, \qquad \{\Theta'_n\}_{n=1}^{\infty}, \qquad \{\Theta_n\}_{n=1}^{\infty}
\end{equation} of $\tilde{\mathbf{U}}^{\imath}$. These elements are defined using recursive formulas and generating functions, and it is difficult to express them in closed form. A similar problem applies to the Baseilhac-Kolb elements of $O_q$. To mitigate this difficulty, we will map everything to $T_q$ using $p$ and $\upsilon$. Our main objective is to express the $p$-images of the Baseilhac-Kolb elements of $O_q$ and the images under \eqref{vp} of the elements \eqref{boom} in the basis for $T_q$ and also in an attractive closed form.  The main results of this paper are Theorems \ref{chainofmen}, \ref{whyyyyy}, \ref{TP1}, \ref{theorem}, \ref{thetathm}, \ref{hprime}, \ref{hzero}.

\medskip 

The paper is organized as follows. In Section 2, we list some notation used throughout the paper. In Section 3, we recall some basic facts about $O_q$. In Section 4, we recall what is known about $\tilde{\mathbf{U}}^{\imath}.$ In Sections 5 and 6, we discuss the Baseilhac-Kolb elements of $O_q$ together with the $\upsilon$-images of the elements \textcolor{black}{in} \eqref{boom}. In Sections 7 and 8, we recall the algebra $T_q$ and the homomorphism $p: O_q \mapsto T_q$. In Section 9, we give the $p$-images of the Baseilhac-Kolb elements and the images under \eqref{vp} of the elements $\{B_{1,r}\}_{r \in \mathbb{Z}}$. In Sections 10--12, we give the images under \eqref{vp} of the elements $\{\Theta'_n\}_{n=1}^{\infty}$ and $\{\Theta_n\}_{n=1}^{\infty}$. In Sections 13 and 14, we give the images under \eqref{vp} of the elements $\{H'_n\}_{n=1}^{\infty}$ and $\{H_n\}_{n=1}^{\infty}$. Finally, in Section 15, we revisit some identities mentioned in Section 6 and describe what they look like at the level of $T_q$.

\section{Notation}

In this section, we list some notation conventions that are used throughout the paper.

\begin{itemize}

\item Let ${K}$ be a field of characteristic $0$.

\item All algebras in this paper are associative, unital, and over $K$.

\item For elements $u,v$ in any algebra, define $[u,v]=uv-vu$. We call $[u,v]$ the \textit{commutator} of $u$ and $v$. For nonzero $r \in K$, define $[u,v]_r=ruv-r^{-1}vu$. We call $[u,v]_r$ the \textit{r-commutator} of $u$ and $v$.

\item Fix nonzero $q \in K$. We assume $q$ is not a root of unity. Fix a square root of $q$, which we call $q^{1/2}$. 





\item For $n \ge 0$, define $[n]_q=\frac{q^n-q^{-n}}{q-q^{-1}}$. We are using notation from \cite{T1}.

\item By an \textit{automorphism} of an algebra $\mathcal A$, we mean an algebra isomorphism from $\mathcal A$ to $\mathcal A$.

\item The \textit{opposite algebra} of an algebra $\mathcal A$, denoted by$\mathcal A^{op}$, is the algebra on the vector space $\mathcal A$ such that for $a,b \in \mathcal A^{op}$, $ab$ (in $\mathcal A^{op}$) $=ba$ (in $\mathcal A$). Note that if $\mathcal A$ is commutative, then $\mathcal A = \mathcal A^{op}$.

\item By an \textit{antiautomorphism} of an algebra $\mathcal{A}$, we mean an algebra isomorphism from $\mathcal A$ to $\mathcal A^{op}$. 

\item Throughout this paper, $t$ is an indeterminate.

\item For any integer $m$, we define $\delta_{m,ev}$ to equal $1$ if $m$ is even and $0$ if $m$ is odd.
\end{itemize}

\section{The $q$-Onsager algebra and the universal $q$-Onsager algebra}

In this section, we recall the $q$-Onsager algebra $O_q$ (see \cite[Section 3]{T1}). We also recall the universal $q$-Onsager algebra $\tilde{\mathbf{U}}^{\imath}$ (see \cite[Definition 2.1]{Drinfeld}).

\begin{definition} \label{oq} \rm(See \cite[Lemma 5.4]{assocscheme}.)
Let $O_q$ denote the algebra defined by generators $W_0$, $W_1$ and relations
\begin{align}
[{W}_0,[{W}_0,[{W}_0,{W}_1]_q]_{q^{-1}}] &= -(q^2-q^{-2})^2[{W}_0,{W}_1], \label{eq:2p1a} \\ 
[{W}_1,[{W}_1,[{W}_1,{W}_0]_q]_{q^{-1}}] &= -(q^2-q^{-2})^2[{W}_1,{W}_0].
\label{eq:2p2a} 
\end{align}
We call $O_q$ the $q$-\textit{Onsager algebra}. We call \eqref{eq:2p1a} and \eqref{eq:2p2a} the $q$-\textit{Dolan-Grady relations.}
\end{definition}

We present an automorphism and an antiautomorphism of $O_q$.

\begin{lemma} {\rm (See \cite[Lemma 3.3]{T2}.)} \label{molybdenum}
There exists an automorphism $\sigma$ of $O_q$ that sends
$$W_0 \mapsto W_1, \qquad \qquad W_1 \mapsto W_0.$$

\end{lemma}

\begin{lemma} {\rm (See \cite[Lemma 3.4]{T2}.)} \label{technetium}
There exists an antiautomorphism $\dagger$ of $O_q$ that sends
$$W_0 \mapsto W_0, \qquad \qquad W_1 \mapsto W_1.$$
\end{lemma}

\begin{lemma} \label{sigmadagger}
With reference to Lemmas \ref{molybdenum} and \ref{technetium}, we have $\sigma \circ \dagger = \dagger \circ \sigma$ in $O_q$. That is to say, this diagram commutes:
\begin{center}
\begin{tikzcd}
O_q \arrow[r,"\dagger"] \arrow[d,"\sigma"'] & O_q \arrow[d,"\sigma"] \\
O_q \arrow[r,"\dagger"'] & O_q 
\end{tikzcd}
\end{center}
\end{lemma}

\begin{proof}
Chase $W_0$ and $W_1$ around the diagram.
\end{proof}

In Lemma \ref{sigmadagger} we looked at the composition $\dagger \circ \sigma$. This composition will be important later in the paper. For now, we note that it is an antiautomorphism of $O_q$ that interchanges $W_0$ and $W_1$.

In \cite{Drinfeld}, Lu and Wang introduced the universal $q$-Onsager algebra $\tilde{\mathbf U}^{\imath}$.

\begin{definition} \label{intro4} {\rm (See \cite[Definition 2.1]{Drinfeld}.)}
\rm Let $\tilde{\mathbf{U}}^{\imath}$ denote the algebra defined by generators $B_0$, $B_1$, $\mathbb{K}_0^{\pm 1}$, $\mathbb{K}_1^{\pm 1}$ and the following relations:
\begin{equation}\mathbb{K}_1\mathbb{K}_1^{-1} = 1 = \mathbb{K}_1^{-1}\mathbb{K}_1, \qquad \mathbb{K}_0\mathbb{K}_0^{-1} = 1 = \mathbb{K}_0^{-1}\mathbb{K}_0, \nonumber \end{equation}
\begin{equation} \mathbb{K}_0, \mathbb{K}_1 \text{ are central},\nonumber \end{equation} 
\begin{align}
[{B}_0,[{B}_0,[{B}_0,{B}_1]_q]_{q^{-1}}] &= -q^{-1}{(q+q^{-1})^2}[{B}_0,{B}_1]\mathbb{K}_0, \nonumber 
\\ 
[{B}_1,[{B}_1,[{B}_1,{B}_0]_q]_{q^{-1}}] &= -q^{-1}{(q+q^{-1})^2}[{B}_1,{B}_0]\mathbb{K}_1. \nonumber
\end{align}
\end{definition}

The algebra $\tilde{\mathbf{U}}^{\imath}$ is known as the \textit{universal $q$-Onsager algebra.}

\begin{lemma}\label{upsilon} {\rm (See \cite[Remark 5.7]{T2}.)}
There exists a surjective algebra homomorphism $\upsilon: \tilde{\mathbf{U}}^{\imath} \mapsto O_q$ that sends
$$B_0 \mapsto \frac{W_0}{q^{1/2}(q-q^{-1})}, \qquad B_1 \mapsto \frac{W_1}{q^{1/2}(q-q^{-1})}, \qquad \mathbb{K}_0 \mapsto 1, \qquad \mathbb{K}_1 \mapsto 1.$$

Moreover, the map $\upsilon$ is surjective.
\end{lemma}

Next, we give an automorphism and an antiautomorphism of $\tilde{\mathbf{U}}^{\imath}$. 
\begin{lemma} {\rm (See \cite[Section 2]{Drinfeld}.)} \label{s} There exists an automorphism $\Phi$ of $\tilde{\mathbf{U}}^{\imath}$ that sends
\begin{equation}
B_0 \mapsto B_1, \qquad B_1 \mapsto B_0, \qquad \mathbb{K}_0 \mapsto \mathbb{K}_1, \qquad \mathbb{K}_1 \mapsto \mathbb{K}_0. \nonumber
\end{equation}
\end{lemma}

\begin{lemma} {\rm (See \cite[Section 3]{Drinfeld}.)} \label{prigami} 
There exists an antiautomorphism $\flat$ of $\tilde{\mathbf{U}}^{\imath}$ that fixes $B_0, B_1, \mathbb{K}_0, \mathbb{K}_1$.
\end{lemma}

\begin{lemma}
The following diagram commutes:

\begin{center}
\begin{tikzcd}
\tilde{\mathbf{U}}^{\imath} \arrow[r,"\flat"] \arrow[d,"\Phi"'] & \tilde{\mathbf{U}}^{\imath}  \arrow[d,"\Phi"] \\
\tilde{\mathbf{U}}^{\imath} \arrow[r,"\flat"'] &  \tilde{\mathbf{U}}^{\imath} 
\end{tikzcd}
\end{center}
\end{lemma}

\begin{proof}
Chase $B_0$, $B_1$, $\mathbb{K}_0$, $\mathbb{K}_1$ around the diagram.

\end{proof}

\begin{lemma} \label{overwalk}
The following diagram commutes:

\begin{center}
\begin{tikzcd}
\tilde{\mathbf{U}}^{\imath} \arrow[r,"\Phi"] \arrow[d,"\upsilon"'] & \tilde{\mathbf{U}}^{\imath}  \arrow[d,"\upsilon"] \\
O_q \arrow[r,"\sigma"'] &  O_q 
\end{tikzcd}
\end{center}

\end{lemma}

\begin{proof}
By Lemmas \ref{upsilon} and \ref{s}, and since $\sigma$ swaps $W_0$ and $W_1$.
\end{proof}

\begin{lemma} \label{underwalk}
The following diagram commutes:

\begin{center}
\begin{tikzcd}
\tilde{\mathbf{U}}^{\imath} \arrow[r,"\flat"] \arrow[d,"\upsilon"'] & \tilde{\mathbf{U}}^{\imath}  \arrow[d,"\upsilon"] \\
O_q \arrow[r,"\dagger"'] &  O_q 
\end{tikzcd}
\end{center}

\end{lemma}

\begin{proof}
By Lemmas \ref{upsilon} and \ref{prigami}, and since $\dagger$ fixes $W_0$ and $W_1$.
\end{proof}

\section{Some elements of $\tilde{\mathbf{U}}^{\imath}$}

In Definition \ref{intro4}, we recalled the generators $B_0, B_1, \mathbb{K}_0^{\pm 1}, \mathbb{K}_1^{\pm 1}$ of the algebra $\tilde{\mathbf{U}}^{\imath}$. In this section we review some additional elements of $\tilde{\mathbf{U}}^{\imath}$ introduced in \cite{Drinfeld}.


Following \cite{Drinfeld}, we define
\begin{equation} \label{0}
\Theta'_0 = \frac{1}{q-q^{-1}}, \qquad \qquad \Theta_0 = \frac{1}{q-q^{-1}}, 
\end{equation}
\begin{equation} \label{1} \Theta'_1 = q^2B_0B_1-B_1B_0, \qquad \qquad \Theta_1= q^2B_0B_1-B_1B_0.\end{equation}
For notational convenience, we define $\Theta'_n=0$ and $\Theta_n=0$ for $n<0$.
Further define $\mathbb{K}_{\delta} = \mathbb{K}_0\mathbb{K}_1$.

\begin{lemma} \label{troublemkr}
The composition $\flat \circ \Phi$ fixes $\Theta_1$.
\end{lemma}

\begin{proof}
Follows from Lemmas \ref{s} and \ref{prigami}.
\end{proof}

\noindent Next, we recall the elements $\{B_{1,r}\}_{r \in \mathbb{Z}}$ of $\tilde{\mathbf{U}}^{\imath}$ from \cite[Section 2]{Drinfeld}. These elements satisfy
\begin{equation} B_{1,0} = B_1, \qquad \qquad B_{1,-1} = B_0\mathbb{K}_0^{-1},\end{equation}
and for $\ell \in \mathbb{Z}$,
\begin{equation} \label{sensitivity} [\Theta_1,B_{1,\ell}] = (q+q^{-1})\left(B_{1,\ell+1}-B_{1,\ell-1}\mathbb{K}_{\delta}\right). \end{equation}

For $n \ge 2$, we define $\Theta'_n$ and $\Theta_n$ by the following equations:
\begin{align}
\Theta'_n &= \bigl(-B_{1,n-1}B_{0}+q^2B_{0}B_{1,n-1}+(q^2-1)\sum_{\ell=0}^{n-2}B_{1,\ell}B_{1,n-\ell-2} \bigr)\mathbb{K}_0,
\label{x1}\\
\Theta_n &= \Theta'_n  - \delta_{n,ev} q^{1-n}\mathbb{K}_{\delta}^{n/2}- \sum_{\ell=1}^{\lfloor \frac{n-1}{2} \rfloor} (q^2-1)q^{-2\ell} \Theta'_{n-2\ell} \mathbb{K}_{\delta}^{\ell}. \label{3} \end{align}

We recall the exponential function \begin{equation} \label{exp}\exp(z) = \sum_{n=0}^{\infty} \frac{z^n}{n!}.\end{equation}

We define the following generating functions for $\tilde{\mathbf{U}}^{\imath}$:

\begin{align} \label{0T1t} \Theta'(t) &= (q-q^{-1})\sum_{n=0}^{\infty} \Theta'_n t^n, \\ \label{0tt} \Theta(t) &= (q-q^{-1})\sum_{n=0}^{\infty} \Theta_n t^n. \end{align}

Next, we define the generating functions $H(t)$ and $H'(t)$ by
\begin{equation} \label{iamgrandma}\exp\left((q-q^{-1})H'(t)\right) = \Theta'(t), \end{equation}
\begin{equation} \label{iamold} \exp\left((q-q^{-1})H(t)\right) = \Theta(t). \end{equation}

By \eqref{exp} and since the constant term of $\Theta'(t)$ is $1$, the constant term of $H'(t)$ is $0$. By a similar argument, the constant term of $H(t)$ is $0$.

\medskip
For $n \ge 1$, we define $H'_n$ and $H_n$ by the following equations:
\begin{align}
H'(t) = \sum_{n=1}^{\infty} H'_nt^n, \label{prethat} \\
H(t) = \sum_{n=1}^{\infty} H_nt^n. \label{prethis} 
\end{align}

\begin{lemma} {\rm (See \cite[Propositions 2.3 and 2.11]{Drinfeld}.)} \label{lemma32}
The following equations hold in $\tilde{\mathbf{U}}^{\imath}$. For $m,n \ge 1$,

\begin{equation} \label{121}
[\Theta'_m, \Theta'_n]=0, \qquad [\Theta_m, \Theta_n]=0, \qquad [H_m,H_n]=0.
\end{equation}

\end{lemma}

\begin{proposition} \label{bluegold}
The following equations hold in $\tilde{\mathbf{U}}^{\imath}$. For $m \ge 1$ and $r, s \in \mathbb{Z},$

\begin{align}
&[H_m',B_{1,r}] = [H_m,B_{1,r}] = \frac{[2m]_q}{m}B_{1,r+m}-\frac{[2m]_q}{m}B_{1,r-m}\mathbb{K}_{\delta}^m, 
\label{tasty_b}\\
&[\Theta'_m,B_{1,r}]+[\Theta'_{m-2},B_{1,r}]\mathbb{K}_{\delta} = [\Theta'_{m-1},B_{1,r+1}]_{q^2} + [\Theta'_{m-1},B_{1,r-1}]_{q^{-2}}\mathbb{K}_{\delta}, \label{awesomesauceprime}\\
&[\Theta_m,B_{1,r}]+[\Theta_{m-2},B_{1,r}]\mathbb{K}_{\delta} = [\Theta_{m-1},B_{1,r+1}]_{q^2} + [\Theta_{m-1},B_{1,r-1}]_{q^{-2}}\mathbb{K}_{\delta},\label{awesomesauce}\\
&q[B_{1,r},B_{1,s+1}]_q-q[B_{1,r+1},B_{1,s}]_{q^{-1}}= \nonumber \\ &\Theta_{s-r+1}\mathbb{K}_{\delta}^r\mathbb{K}_1 -q^{-2}\Theta_{s-r-1}\mathbb{K}_{\delta}^{r+1}\mathbb{K}_1   +\Theta_{r-s+1}\mathbb{K}_{\delta}^s\mathbb{K}_1 - q^{-2}\Theta_{r-s-1}\mathbb{K}_{\delta}^{s+1}\mathbb{K}_1.
\end{align}
\end{proposition}
\begin{proof}
The above equations are reformulations of \cite[Eqns. 2.35, 2.33, 2.21, 2.23, 2.34]{Drinfeld}. 
\end{proof}

\begin{lemma} {\rm (See \cite[Lemma 2.9]{Drinfeld}.)} \label{t1} In $\tilde{\mathbf U}^{\imath}$, we have

\begin{equation} \Theta(t) = \frac{1-\mathbb{K}_{\delta} t^2}{1-q^{-2}\mathbb{K}_{\delta} t^2}\Theta'(t). \label{tbt} \end{equation}
\end{lemma}

\begin{lemma} \label{1m1m1}
For $r \in \mathbb{Z}$, the composition $\flat \circ \Phi$ sends

\begin{equation}\label{iheartsabrina}B_{1,r} \mapsto B_{1,-r-1}.\end{equation}
\end{lemma}

\begin{proof}

It suffices to show that $\flat \circ \Phi$ sends

\begin{equation} \label{hannah}
B_{1,r} \mapsto B_{1,-r-1}, \qquad B_{1,-r-1} \mapsto B_{1,r}, \qquad r \ge 0.
\end{equation}

We prove \eqref{hannah} by induction on $r$.  The case of $r=0$ is routine, so we assume $r \ge 1$. By induction, $\flat \circ \Phi$ sends $$B_{1,k} \mapsto B_{1,-k-1}, \qquad B_{1, -k-1} \mapsto B_{1,k}$$ for $0 \le k \le r-1$. Our goal is to show that $\flat \circ \Phi$ sends $$B_{1,r} \mapsto B_{1, -r-1}, \qquad B_{1,-r-1} \mapsto B_{1,r}.$$ We let $\ell=r-1$ in \eqref{sensitivity} and rearrange terms to obtain

\begin{equation} \label{arxivisdown} B_{1,r}=B_{1,r-2}+\frac{[\Theta_1,B_{1,r-1}]}{q+q^{-1}}. \end{equation}

By Lemma  \ref{troublemkr}, $\flat \circ \Phi$ fixes $\Theta_1$. We apply $\flat \circ \Phi$ to both sides of \eqref{arxivisdown} and we see that $\flat \circ \Phi$ sends $B_{1,r} \mapsto B_{1,-r-1}$. 

Similarly, we let $\ell = -r$ in \eqref{sensitivity} and rearrange terms to obtain

\begin{equation} \label{arxivisdown2} B_{1,-r+1}=B_{1,-r-1}+\frac{[\Theta_1,B_{1,-r}]}{q+q^{-1}}. \end{equation}

We apply $\flat \circ \Phi$ to both sides of \eqref{arxivisdown2} and we see that $\flat \circ \Phi$ sends $B_{1, -r-1} \mapsto B_{1,r}$. Therefore, \eqref{iheartsabrina} holds for all integers $r$.
\end{proof}










\section{The Baseilhac-Kolb elements of $O_q$}

In this section, we recall some elements of $O_q$ known as the \text{Baseilhac-Kolb elements.}

\begin{definition} \label{bkd}
{\rm (See \cite[Section 3]{p}.) In the algebra $O_q$, we define the elements} \\ \begin{equation}\{{B}_{n\delta+\alpha_0}\}_{n=0}^{\infty}, \qquad \{{B}_{n\delta+\alpha_1}\}_{n=0}^{\infty}, \qquad \{{B}_{n\delta}\}_{n=1}^{\infty}\label{bk} \end{equation} {\rm in the following way:}

\begin{align}
{B}_{\delta}  &= q^{-2}{W}_1{W}_0-{W}_0{W}_1, \label{15} \\
{B}_{\alpha_0} &= {W}_0, \label{300} \\
{B}_{\delta+\alpha_0} &= {W}_1 + \frac{q[{B}_{\delta},{W}_0]}{(q-q^{-1})(q^2-q^{-2})}, \label{rowan} \\
{B}_{n\delta+\alpha_0} &= {B}_{(n-2)\delta+\alpha_0} + \frac{q[{B}_{\delta},{B}_{(n-1)\delta+\alpha_0}]}{(q-q^{-1})(q^2-q^{-2})}, \qquad n \ge 2, \label{raccoon} \\
{B}_{\alpha_1} &= {W}_1, 
\label{311}\\
{B}_{\delta+\alpha_1} &= {W}_0 - \frac{q[{B}_{\delta},{W}_1]}{(q-q^{-1})(q^2-q^{-2})}, \label{regdab} \\
{B}_{n\delta+\alpha_1} &= {B}_{(n-2)\delta+\alpha_1} -\frac{q[{B}_{\delta},{B}_{(n-1)\delta+\alpha_1}]}{(q-q^{-1})(q^2-q^{-2})} ,\qquad n \ge 2, \label{badger}  \\
{B}_{n\delta} &= q^{-2}{B}_{(n-1)\delta+\alpha_1}{W}_0 - {W}_0 {B}_{(n-1)\delta+\alpha_1} \nonumber \\ & \quad + (q^{-2}-1)\sum_{\ell=0}^{n-2} {B}_{\ell\delta+\alpha_1}{B}_{(n-\ell-2)\delta+\alpha_1} \label{bdelta}, \qquad n \ge 2.
\end{align}

{\rm We call the elements in \eqref{bk} the \textit{Baseilhac-Kolb} elements of $O_q$. For notational convenience, we define $B_{0\delta} = q^{-2}-1$.}

\end{definition}

By \cite[Proposition 5.12]{p}, the elements $\{{B}_{n\delta}\}_{n=1}^{\infty}$ mutually commute. 

\begin{lemma} \label{dagsigbd}
The composition $\dagger \circ \sigma$ fixes $B_{\delta}$.
\end{lemma}

\begin{proof}
Follows from \eqref{15}.
\end{proof}
\begin{proposition} \label{a1a0} For $n \ge 0$, the composition $\dagger \circ \sigma$ sends 

\begin{equation} \label{what} B_{n\delta+\alpha_0} \mapsto B_{n\delta+\alpha_1}, \qquad B_{n\delta+\alpha_1} \mapsto B_{n\delta+\alpha_0}.\end{equation}

\end{proposition}

\begin{proof}
We use induction on $n$. The case of $n=0$ follows from \eqref{300}, \eqref{311}, and the comment below Lemma \ref{sigmadagger}. The case of $n=1$ follows from \eqref{rowan}, \eqref{regdab},  and Lemma \ref{dagsigbd}.

Assume $n \ge 2$. By induction, \eqref{what} holds for $n-1$ and $n-2$. We apply $\dagger \circ \sigma$ to equations \eqref{raccoon} and \eqref{badger} and we obtain \eqref{what}.
\end{proof}

\section{Various $\upsilon$-images in $O_q$}
Recall the map $\upsilon: \tilde{\mathbf{U}}^{\imath} \mapsto O_q$ from Lemma \ref{upsilon}. In this section, we give the $\upsilon$-images of various elements of $\tilde{\mathbf{U}}^{\imath}$. 

In the following result, we apply the map $\upsilon$ to the elements $\{B_{1,r}\}_{r \in \mathbb{Z}}$ of $\tilde{\mathbf{U}}^{\imath}$. 

\begin{proposition} \label{arizz}
For $n \ge 0$, the map $\upsilon$ sends
\begin{equation} \label{5553} B_{1,-
n-1} \mapsto \frac{B_{n\delta+\alpha_0}}{q^{1/2}(q-q^{-1})}, \end{equation}
\begin{equation} \label{5530} B_{1,n} \mapsto \frac{B_{n\delta+\alpha_1}}{q^{1/2}(q-q^{-1})}. \end{equation}
\end{proposition}

\begin{proof}

We prove \eqref{5530} by induction on $n$. The base cases are $n=0$ and $n=1$.

By construction, $B_{0\delta+\alpha_1} = W_1$ and $B_{1,0} = B_1$, so \eqref{5530} holds for $n=0$ by Theorem \ref{upsilon}. 

For $n=1$, \eqref{5530} follows from \eqref{regdab}, \eqref{1}, and \eqref{sensitivity}.

\medskip

For the inductive step, we assume $n \ge 2$. We let $\ell=n-1$ in \eqref{sensitivity} and rearrange terms to obtain

\begin{equation} B_{1,n}=B_{1,n-2}+\frac{[\Theta_1,B_{1,n-1}]}{q+q^{-1}}.  \nonumber \end{equation}

We apply $\upsilon$ to both sides of \eqref{arxivisdown} and we obtain \eqref{badger}. Hence, by induction, since \eqref{5530} holds for $n=0$ and $n=1$, it holds for all $n \ge 0$.

Equation \eqref{5553} follows from Lemma \ref{1m1m1} and Proposition \ref{a1a0}.
\end{proof}

\begin{proposition} \label{thetabnd}
For $n \ge 0$, the map $\upsilon$ sends

\begin{equation} \label{tbd} \Theta'_{n} \mapsto -\frac{qB_{n\delta}}{(q-q^{-1})^2}. \end{equation}
\end{proposition}

\begin{proof}
For $n=0$, \eqref{tbd} follows from the definitions of $\Theta'_0$ and $B_{0\delta}$. For $n=1$, \eqref{tbd} follows from \eqref{1} and Lemma \ref{upsilon}. For $n\ge 2$,
we apply the map $\upsilon$ to both sides of \eqref{x1} and we see that $\upsilon(\Theta'_n)$ equals $-\frac{q}{(q-q^{-1})^2}$ times the right side of \eqref{bdelta} and \eqref{tbd} follows.
\end{proof}

For notational convenience, for the elements
$$\{\Theta'_n\}_{n=0}^{\infty}, \qquad \{\Theta_n\}_{n=0}^{\infty}, \qquad \{H'_n\}_{n=1}^{\infty}, \qquad \{H_n\}_{n=1}^{\infty}, \qquad \{B_{1,r}\}_{r \in \mathbb{Z}}$$

of $\tilde{\mathbf{U}}^{\imath},$ we retain the same notation for their $\upsilon$-images in $O_q$. For instance, in $O_q$, we have

\begin{equation} \label{17w} \Theta'_n = -\frac{qB_{n\delta}}{(q-q^{-1})^2}, \qquad n \ge 0. \end{equation}

In a similar fashion, for the generating functions

$$\Theta'(t), \qquad \Theta(t), \qquad H'(t), \qquad H(t)$$

of $\tilde{\mathbf{U}}^{\imath}$, we retain the same notation for their $\upsilon$-images in $O_q$. 

By construction and \eqref{iamgrandma}, \eqref{iamold} the following equations hold in $O_q$:
\begin{equation}
\exp\left((q-q^{-1})H'(t)\right) = \Theta'(t), \qquad \qquad
\exp\left((q-q^{-1})H(t)\right) = \Theta(t). \label{this} 
\end{equation}

The equations about $\tilde{\mathbf{U}}^{\imath}$ from Section 4 still hold in $O_q$. We present two examples for future use.
\begin{lemma} In $O_q$, for $n \ge 1$,
\begin{equation}
\Theta_n = \Theta'_n - \delta_{n,ev} q^{1-n} - \sum_{\ell=1}^{\lfloor \frac{n-1}{2} \rfloor} (q^2-1)q^{-2\ell} \Theta'_{n-2\ell}.\label{642} \end{equation}
\end{lemma}

\begin{proof}
Apply $\upsilon$ to both sides of \eqref{3}.
\end{proof}
Define
\begin{equation} \nonumber
\Theta'_0 = \frac{1}{q-q^{-1}}, \qquad \qquad \Theta_0 = \frac{1}{q-q^{-1}}, 
\end{equation}
and for notational convenience, we define $\Theta'_n=0$ and $\Theta_n=0$ for $n<0$.

\begin{lemma} \label{lemma}
In $O_q$,
\begin{equation} \Theta(t) = \frac{1-t^2}{1-q^{-2} t^2}\Theta'(t).
\label{girlinred} \end{equation}

\end{lemma}

\begin{proof}
Apply $\upsilon$ to both sides of \eqref{tbt}.
\end{proof}

\begin{proposition} \label{redgoldgreen}
The following equations hold in $O_q$. For $m,n \ge 1$ and $r, s \in \mathbb{Z},$

\begin{align}
&[\Theta'_m, \Theta'_n]=0, \qquad [\Theta_m, \Theta_n]=0, \qquad [H_m,H_n]=0, \label{hhoq} \\
&[H_m',B_{1,r}] = [H_m,B_{1,r}] = \frac{[2m]_q}{m}B_{1,r+m}-\frac{[2m]_q}{m}B_{1,r-m}, 
\label{tasty_boq}\\
&[\Theta'_m,B_{1,r}]+[\Theta'_{m-2},B_{1,r}] = [\Theta'_{m-1},B_{1,r+1}]_{q^2} + [\Theta'_{m-1},B_{1,r-1}]_{q^{-2}}, \label{awesomesauceprimeoq}\\
&[\Theta_m,B_{1,r}]+[\Theta_{m-2},B_{1,r}] = [\Theta_{m-1},B_{1,r+1}]_{q^2} + [\Theta_{m-1},B_{1,r-1}]_{q^{-2}},\label{awesomesauceoq}\\
\begin{split} \label{unlabeled1}
&q[B_{1,r},B_{1,s+1}]_q-q[B_{1,r+1},B_{1,s}]_{q^{-1}} \\ &= \Theta_{s-r+1} -q^{-2}\Theta_{s-r-1}  +\Theta_{r-s+1} - q^{-2}\Theta_{r-s-1}. 
\end{split}
\end{align}
\end{proposition}
\begin{proof}
Apply $\upsilon$ to everything in Lemma \ref{lemma32} and Proposition \ref{bluegold}.
\end{proof}

\section{The quantum torus $T_q$}

In this section, we consider an algebra $T_q$, called the \text{quantum torus}. We review some properties of $T_q$ and display a basis for the vector space $T_q$. 

\begin{definition} 
\rm (See \cite{Magic}.)
Define the algebra $T_q$ by generators
$$x,y,x^{-1},y^{-1}$$
and relations
\begin{align}
xx^{-1}=1=x^{-1}x, \qquad \qquad
yy^{-1} = 1 = y^{-1}y, \qquad \qquad xy=q^2yx. \label{aldo}
\end{align}

The algebra $T_q$ is called the \textit{quantum torus}. 
\end{definition}

\begin{lemma} \label{lem:3facts}
The following relations hold in $T_q$:
\begin{align}
xy = q^2yx, \qquad & \qquad
x^{-1}y = q^{-2}yx^{-1}, \nonumber\\
x^{-1}y^{-1} = q^2y^{-1}x^{-1}, \qquad & \qquad
xy^{-1} = q^{-2}y^{-1}x. \label{eq:8p4} \nonumber
\end{align}
\end{lemma}

\begin{proof}
Routine consequence of the relations in \eqref{aldo}.
\end{proof}






We emphasize a point for later use.

\begin{lemma} \label{xyyx} The following elements of $T_q$ mutually commute:
\begin{align*}
xy, \quad yx, \quad x^{-1}y^{-1}, \quad y^{-1} x^{-1}.
\end{align*}
\end{lemma}

\begin{proof}
By Lemma \ref{lem:3facts}.
\end{proof}

We display a basis for the vector space $T_q$.

\begin{lemma}{\rm (See \cite[p.~3]{Magic}.)} \label{sbl}
The following is a basis for the vector space $T_q$:
\begin{equation} \label{sb} x^ay^b, \qquad a,b \in \mathbb Z. \end{equation}
\end{lemma}

We call \eqref{sb} the \textit{standard basis} of $T_q$.

\section{The homomorphism $p : O_q \mapsto T_q$}
In \cite{Owen}, we discussed an algebra homomorphism $p: O_q \mapsto T_q.$ In this section, we review some essential points regarding $p$.

\begin{definition} \label{label} {\rm (See \cite[Definition 4.4]{Owen}.)}
\rm
For the algebra $T_q$, define
\begin{equation} \label{xyz}
    w_0 = x + x^{-1}, \qquad w_1 = y + y^{-1}.
\end{equation}
\end{definition}


\begin{lemma} \label{ww}
{\rm (See \cite[Lemma 4.5]{Owen}.) }The following relations hold in $T_q$:
\begin{align}[{w}_0,[{w}_0,{w}_1]_q]_{q^{-1}} &= -(q^2-q^{-2})^2{w}_1, \label{abbottp} \\
[{w}_1,[{w}_1,{w}_0]_q]_{q^{-1}} &= -(q^2-q^{-2})^2{w}_0. \label{costellop}
\end{align}

\end{lemma}


\begin{corollary} \label{barium}
{\rm (See \cite[Corollary 4.6]{Owen}.) }The following relations hold in $T_q$:
\begin{align}
[w_0,[w_0,[w_0,w_1]_q]_{q^{-1}}] &= -(q^2-q^{-2})^2[w_0,w_1], \label{abbott}\\
[w_1,[w_1,[w_1,w_0]_q]_{q^{-1}}] &= -(q^2-q^{-2})^2[w_1,w_0]. \label{costello}
\end{align}
\end{corollary}


\begin{proposition} {\rm (See \cite[Proposition 4.7]{Owen}.)}\label{p}
There exists an algebra homomorphism $p : O_q \mapsto T_q$ that sends $W_0 \mapsto w_0$ and $W_1 \mapsto w_1$.
\end{proposition}

\section{The $p$-images of the Baseilhac-Kolb elements and $B_{1,r}$ elements}
Recall the map $p: O_q \mapsto T_q$. Our next general goal is to apply $p$ to all of the elements in $O_q$ that we have defined and referenced up to this point. In the next six sections, we consider this topic. 

In this section, we give the $p$-images of the Baseilhac-Kolb elements, as well as the elements $\{B_{1,r}\}_{r \in \mathbb{Z}}$ of $O_q$. In the following sections, we will give the $p$-images of the elements $\{\Theta'_n\}_{n=1}^{\infty}, \{\Theta_n\}_{n=1}^{\infty}, \{H'_n\}_{n=1}^{\infty}, \{H_n\}_{n=1}^{\infty}$. 

\begin{lemma} \label{bdl}
The map $p$ sends \begin{equation}
B_{\delta} \mapsto -(q^{2}-q^{-2})(yx+y^{-1}x^{-1}). \label{smokychicken}
\end{equation}
\end{lemma}

\begin{proof}
Apply $p$ to both sides of \eqref{15}.
\end{proof}
The following is our first main result.

\begin{theorem} \label{chainofmen}
For $n \ge 0$, the map $p$ sends

     \begin{equation}\label{k2}B_{n\delta+\alpha_0} \mapsto x(yx)^n+x^{-1}(y^{-1}x^{-1})^n,\end{equation}
    \begin{equation}\label{k1}B_{n\delta+\alpha_1} \mapsto y(xy)^n+y^{-1}(x^{-1}y^{-1})^n.\end{equation}

\end{theorem}
\begin{proof}

    It will be convenient to prove \eqref{k1} first. We use induction on $n$. The cases of $n=0$ and $n=1$ are routinely verified.
    
    We pick $n \ge 2$. By induction, \eqref{k1} holds for $n-2$ and $n-1$; we will show that \eqref{k1} holds for $n$.

    By substituting \eqref{smokychicken} and \eqref{k1} into \eqref{badger} and simplifying, we see that \eqref{k1} holds for $n$, and thus for all $n \ge 0$.

    We have now proven \eqref{k1}. The proof of \eqref{k2} is analogous and omitted. 
\end{proof}

Next, we describe how the expressions on the right in \eqref{k2} and \eqref{k1} look in the standard basis for $T_q$.

\begin{corollary} \label{doingtherightthing}
For $n \ge 0$, the map $p$ sends
\begin{align}
B_{n\delta+\alpha_0} &\mapsto q^{-n(n-1)}(x^{n+1}y^{n}+x^{-n-1}y^{-n}), \label{4240} \\
B_{n\delta+\alpha_1} &\mapsto q^{-n(n+1)}(x^ny^{n+1}+x^{-n}y^{-n-1}). \label{4241}
\end{align}
\end{corollary}

Next, we apply the map $p$ to the elements $\{B_{1,r}\}_{r \in \mathbb{Z}}$ introduced in Section 5.

\begin{proposition} For $r \in \mathbb{Z}$, the map $p$ sends
\begin{equation} \label{bluebird} B_{1,r} \mapsto \frac{y(xy)^r+y^{-1}(x^{-1}y^{-1})^r}{q^{1/2}(q-q^{-1})}.\end{equation}
\end{proposition}

\begin{proof}
Follows from Theorems \ref{arizz} and \ref{chainofmen}.
\end{proof}

Next, we describe how the expression on the right in \eqref{bluebird} looks in the standard basis for $T_q$.

\begin{corollary} For $r \in \mathbb{Z}$, the map $p$ sends
\begin{equation} \label{1108}
B_{1,r} \mapsto q^{-r(r+1)}\frac{x^ry^{r+1}+x^{-r}y^{-r-1}}{q^{1/2}(q-q^{-1})}.
\end{equation}
\end{corollary}

Next, we apply the map $p$ to the elements $\{B_{n\delta}\}_{n=1}^{\infty}$ of $O_q$.


\begin{proposition} \label{B3}
For $n \ge 1$, the map $p$ sends
\begin{equation}\label{thetaB2} 
\begin{split}
B_{n\delta} \mapsto &(q^{-2}-1)\biggl(q^{-n}[n+1]_q(xy)^n+q^{n}[n+1]_q(xy)^{-n} \\ &+\sum_{\ell=1}^{n-1}(1+q^{4\ell-2n})(xy)^{n-2\ell}\biggr).
\end{split}
\end{equation}
\end{proposition}

\begin{proof}
First assume $n = 1$. Then the result follows from Lemma \ref{bdl}. Next we assume $n \ge 2$.
We apply $p$ to both sides of $\eqref{bdelta}$ and the result follows from Theorem \ref{chainofmen} and routine algebraic manipulation.
\end{proof}
We express Proposition \ref{B3} in a more compact form.
\begin{corollary}

For $n \ge 1$, the map $p$ sends
\begin{equation}\label{thetaprime22} 
\begin{split}
B_{n\delta} \mapsto &(q^{-2}-1)\biggl(q^{-n}[n+1]_q(xy)^n+q^{n}[n+1]_q(xy)^{-n}  \\ &+\frac{(xy)^{n-1}-(xy)^{1-n}}{xy-(xy)^{-1}}+\frac{(yx)^{n-1}-(yx)^{1-n}}{yx-(yx)^{-1}}\biggr).
\end{split}
\end{equation}
\end{corollary}

\begin{proof}

The expression within the summation in \eqref{thetaB2} is the sum of two geometric series. These two series simplify to the last two terms in \eqref{thetaprime22}.
\end{proof}

Next, we describe how the expression on the right in \eqref{thetaB2} looks in the standard basis for $T_q$.

\begin{corollary} For $n \ge 1$, the map $p$ sends
\begin{equation}\label{anthology} 
\begin{split}
B_{n\delta} \mapsto &(q^{-2}-1)\biggl(q^{-n^2}[n+1]_qx^ny^n+ q^{n^2}[n+1]_qx^{-n}y^{-n} \\ &+\sum_{\ell=1}^{n-1}(q^{2\ell-n}+q^{n-2\ell}) q^{-(n-2\ell)^2}x^{n-2\ell}y^{n-2\ell}\biggr).
\end{split}
\end{equation} 

\end{corollary}

\section{The $p$-image of $\{\Theta'_n\}_{n=1}^{\infty}$}
In this section we find the $p$-images of the elements $\{\Theta'_n\}_{n=1}^{\infty}$ of $O_q$.

\begin{theorem} \label{whyyyyy}
For $n \ge 1$, the map $p$ sends
\begin{equation}
\begin{split}\label{thetaprime} 
\Theta'_n \mapsto & \frac{1}{q-q^{-1}}\biggl(q^{-n}[n+1]_q(xy)^n+q^{n}[n+1]_q(xy)^{-n} \\ &  +\sum_{\ell=1}^{n-1}(1+q^{4\ell-2n})(xy)^{n-2\ell}\biggr).
\end{split}
\end{equation}

\end{theorem}
\begin{proof}
Follows from \eqref{17w} and Proposition \ref{B3}.
\end{proof}
We express \eqref{thetaprime} in a more compact form.

\begin{corollary}

For $n \ge 1$, the map $p$ sends
\begin{equation}\label{thetaprime2} 
\begin{split}
\Theta'_n \mapsto &\frac{1}{q-q^{-1}}\biggl(q^{-n}[n+1]_q(xy)^n+q^{n}[n+1]_q(xy)^{-n}  \\ &\left. +\frac{(xy)^{n-1}-(xy)^{1-n}}{xy-(xy)^{-1}}+\frac{(yx)^{n-1}-(yx)^{1-n}}{yx-(yx)^{-1}}\right).
\end{split}
\end{equation}
\end{corollary}

\begin{proof}

The expression within the summation in \eqref{thetaprime} is the sum of two geometric series. These two series simplify to the last two terms in \eqref{thetaprime2}.
\end{proof}

We describe how the expression on the right in \eqref{thetaprime} looks in the standard basis for $T_q$.

\begin{corollary} For $n \ge 1$, the map $p$ sends
\begin{equation}\label{torturedpoetsdepartment} 
\begin{split}
\Theta'_n \mapsto &\frac{1}{q-q^{-1}}\biggl(q^{-n^2}[n+1]_qx^ny^n+ q^{n^2}[n+1]_qx^{-n}y^{-n} \\ &+\sum_{\ell=1}^{n-1}(q^{n-2\ell}+q^{2\ell-n}) q^{-(n-2\ell)^2}x^{n-2\ell}y^{n-2\ell}\biggr).
\end{split}
\end{equation} 

\end{corollary}

\section{The generating functions $\Theta'(t)$ and $\Theta(t)$ and their $p$-images}


In the previous section, we considered the $p$-images of the elements $\{\Theta'_n\}_{n=1}^{\infty}$ of $O_q$. Shortly, we will consider the $p$-images of $\{\Theta_n\}_{n=1}^{\infty}.$ However, before we reach that point, it is convenient to bring in the generating functions $\Theta'(t)$ and $\Theta(t)$ for $O_q$. In this section, we apply the map $p$ to these generating functions. Later on, we will apply the results of this section to obtain the $p$-images of $\{\Theta_n\}_{n=1}^{\infty}.$

We recall the generating function $\Theta'(t)$ for $O_q$.

\begin{theorem} \label{TP1}
The map $p$ sends \begin{equation}\label{TP1q}\Theta'(t) \mapsto \frac{(1-q^{2}t^2)(1-q^{-2}t^2)}{(1-xyt)(1-yxt)(1-x^{-1}y^{-1}t)(1-y^{-1}x^{-1}t)}.\end{equation}
\end{theorem}

\begin{proof}
We invoke Lemma \ref{xyyx} and Theorem \ref{whyyyyy}. By employing partial fraction decomposition, we routinely find that the right side of \eqref{TP1q} is 

\begin{equation} \label{pfd}
1 + \frac{A}{1-xyt} + \frac{B}{1-yxt} + \frac{C}{1-x^{-1}y^{-1}t}+\frac{D}{1-y^{-1}x^{-1}t}
\end{equation}

where

\begin{equation}
\begin{split}
A = -\frac{1-q^{2}(xy)^2}{(1-q^2)(1-(xy)^2)}, \qquad & \qquad B = -\frac{1-q^{-2}(yx)^2}{(1-q^{-2})(1-(yx)^2)},\\
C = \frac{1-q^{-2}(yx)^2}{(1-q^{-2})(1-(yx)^2)}, \qquad & \qquad D = \frac{1-q^2(xy)^2}{(1-q^2)(1-(xy)^2)}.
\end{split} \nonumber
\end{equation}

We expand each term in \eqref{pfd} as a power series in $t$:

\begin{equation}\begin{split}\label{raugaj}
&\frac{1}{1-xyt} = \sum_{n =0}^{\infty} (xy)^n t^n, \qquad \qquad \qquad \frac{1}{1-yxt} = \sum_{n=0}^{\infty} (yx)^n t^n, \\
&\frac{1}{1-x^{-1}y^{-1}t} = \sum_{n=0}^{\infty} (x^{-1}y^{-1})^n t^n,  \qquad \frac{1}{1-y^{-1}x^{-1}t} = \sum_{n=0}^{\infty} (y^{-1}x^{-1})^nt^n.
\end{split}\end{equation}

The result follows from the above comments along with Theorem \ref{whyyyyy} and 
\begin{equation} \nonumber \Theta'(t) = (q-q^{-1})\sum_{n=0}^{\infty} \Theta'_n t^n. \end{equation} \end{proof}
\begin{theorem} \label{theorem} The map $p$ sends
\begin{equation}\label{thetatau}
\Theta(t) \mapsto \frac{(1-t^2)(1-q^2t^2)}{(1-xyt)(1-yxt)(1-x^{-1}y^{-1}t)(1-y^{-1}x^{-1}t)}.
\end{equation}
\end{theorem}

\begin{proof}
By \eqref{girlinred} and Theorem \ref{TP1}.
\end{proof}

\section{The $p$-image of $\{\Theta_n\}_{n=1}^{\infty}$}

In this section, we compute the $p$-images of the elements $\{\Theta_n\}_{n=1}^{\infty}$ of $O_q$.

\begin{theorem} \label{thetathm}
For $n \ge 1$, the map $p$ sends

\begin{equation}\label{theta}\begin{split}\Theta_n \mapsto & [n+1]_q\frac{(qyx)^n+(qyx)^{-n}}{q-q^{-1}} \\ &+ \frac{q+q^{-1}}{q-q^{-1}}q^{1-n}\sum_{\ell=1}^{n-1} (qyx)^{n-2\ell}.\end{split}\end{equation}
\end{theorem}


\begin{proof}
We invoke Theorem \ref{theorem}.
By employing partial fraction decomposition, we write the right side of \eqref{thetatau} as

\begin{equation} \label{pdf}
\Theta(t) = q^2 + \frac{\alpha}{1-xyt} + \frac{\beta}{1-yxt} + \frac{\gamma}{1-x^{-1}y^{-1}t}+\frac{\delta}{1-y^{-1}x^{-1}t} \nonumber
\end{equation}

where

\begin{equation}
\begin{split}
    \alpha = \frac{1}{1-q^{-2}}, \qquad & \qquad \beta = \frac{q^4-q^2(yx)^{2}}{(1-q^2)(1-q^2(yx)^{2})}, \nonumber \\ \gamma = \frac{1}{1-q^{-2}}, \qquad & \qquad \delta = \frac{1-q^2(xy)^2}{(1-q^2)(1-q^2(yx)^2)}.
\end{split}
\end{equation}

The result follows from the above comments along with \eqref{raugaj} and \begin{equation}\Theta(t) = (q-q^{-1})\sum_{n=0}^{\infty} \Theta_n t^n.\label{tietz} \nonumber \end{equation} \end{proof} We express \eqref{theta} in a more compact form.

\begin{corollary}
For $n \ge 1$, the map $p$ sends

\begin{equation} \begin{split} \Theta_n & \mapsto \frac{[n+1]_q}{q-q^{-1}}\frac{(qyx)^{n+1}-(qyx)^{-n-1}}{qyx-(qyx)^{-1}} \\ & \quad - \frac{q^2[n-1]_q}{q-q^{-1}}\frac{(qyx)^{n-1} - (qyx)^{1-n}}{qyx-(qyx)^{-1}}. \end{split} \label{taiman2} \end{equation}
\end{corollary}
We describe how \eqref{theta} looks in the standard basis for $T_q$.

\begin{corollary} For $n \ge 1$, the map $p$ sends
\begin{equation}\begin{split}
\Theta_n \mapsto &q^{-n^2}[n+1]_q\frac{x^ny^n+x^{-n}y^{-n}}{q-q^{-1}}
\\&+\frac{q+q^{-1}}{q-q^{-1}}q^{1-n}\sum_{\ell = 1}^{n-1} q^{-(n-2\ell)^2}x^{n-2\ell}y^{n-2\ell}. \end{split} \label{91} \end{equation}

\end{corollary}


\section{The $p$-images of $H'(t)$ and $H(t)$ in $T_q$}

In this section, we give the $p$-images of the generating functions $H'(t)$ and $H(t)$ for $O_q$. We recall the natural logarithm

\begin{equation}\label{log}
\ln(1-z) = -\sum_{n=1}^{\infty} \frac{z^n}{n}.
\end{equation}

\begin{theorem} \label{hprime}
The map $p$ sends $H'(t) \mapsto $

\begin{equation} \begin{split} \label{tija}
&\frac{\ln(1-xy)+\ln(1-yx)+\ln(1-x^{-1}y^{-1})}{q-q^{-1}}\\ &+\frac{\ln(1-y^{-1}x^{-1}) - \ln(1-q^2t^2) - \ln(1-q^{-2}t^2)}{q-q^{-1}}. \end{split}
\end{equation}

\end{theorem}

\begin{proof}
Follows from \eqref{this} and Theorem \ref{TP1}.
\end{proof}

\begin{theorem} \label{hzero}
The map $p$ sends $H(t) \mapsto $

\begin{equation} \begin{split}\label{natalee}
&\frac{\ln(1-xy)+\ln(1-yx)+\ln(1-x^{-1}y^{-1})}{q-q^{-1}} \\ &+\frac{\ln(1-y^{-1}x^{-1}) - \ln(1-q^2t^2) - \ln(1-t^2)}{q-q^{-1}}. \end{split}
\end{equation}

\end{theorem}

\begin{proof}
Follows from \eqref{this} and Theorem \ref{theorem}.
\end{proof}

\section{The $p$-images of $\{H'_n\}_{n=1}^{\infty}$ and $\{H_n\}_{n=1}^{\infty}$}

In this section, we give the $p$-images of the elements $\{H'_n\}_{n=1}^{\infty}$ and $\{H_n\}_{n=1}^{\infty}$ of $O_q$. We will consider the cases of $n$ even and $n$ odd separately.
\begin{theorem} \label{arielle}
For $n \ge 1$, we describe the $p$-image of the element $H'_n$ of $O_q$. For $n$ odd, the map $p$ sends

\begin{equation} \label{H1pt5} H'_n\mapsto \frac{(q^n+q^{-n})\bigl((qyx)^n+(qyx)^{-n}\bigr)}{n(q-q^{-1})}.\end{equation}

For $n$ even, $p$ sends

\begin{equation} \label{H2} H'_n \mapsto \frac{(q^n+q^{-n})\bigl((qyx)^n+(qyx)^{-n}\bigr)-2q^n-2q^{-n}}{n(q-q^{-1})}.\end{equation}
\end{theorem}

\begin{proof}
We define

\begin{equation}h'_n = \begin{cases}
\text{the right side of \eqref{H1pt5},} & \text{$n$ odd} \\
\text{the right side of \eqref{H2},} & \text{$n$ even.} \\
\end{cases}\end{equation}



It suffices to show that $p$ sends $H'(t) \mapsto \displaystyle\sum_{n \ge 1}h'_nt^n.$

We have

$$\sum_{n \ge 1} h'_nt^n = \sum_{n \ge 1}\frac{(1+q^{-2n})(xy)^n+(1+q^{-2n})(x^{-1}y^{-1})^n}{n(q-q^{-1})}t^n - \sum_{\ell \ge 1} \frac{q^{2\ell}+q^{-2\ell}}{\ell(q-q^{-1})}t^{2\ell} \nonumber.$$

In the above equation we evaluate the right side using \eqref{log} and this gives

\begin{equation} \label{thisisnotanequation} 
\begin{split}
\sum_{n \ge 1} h'_nt^n &=\frac{\ln(1-xy)+\ln(1-yx)+\ln(1-x^{-1}y^{-1})}{q-q^{-1}} \\ &\quad + \frac{\ln(1-y^{-1}x^{-1})-\ln(1-q^2t^2) - \ln(1-q^{-2}t^2)}{q-q^{-1}}.
\end{split}
\end{equation}

The result follows from Theorem \ref{hprime}.
\end{proof}

\begin{theorem} \label{isabelle}
For $n \ge 1$, we describe the $p$-image of the element $H_n$ of $O_q$. For $n$ odd, the map $p$ sends

\begin{equation} \label{H1} H_n \mapsto \frac{(q^n+q^{-n})\bigl((qyx)^n+(qyx)^{-n}\bigr)}{n(q-q^{-1})}.\end{equation}

For $n$ even, $p$ sends

\begin{equation}\label{H3} H_n \mapsto \frac{(q^n+q^{-n})\bigl((qyx)^n+(qyx)^{-n}\bigr)-2q^n-2}{n(q-q^{-1})}.\end{equation}
    
\end{theorem}
\begin{proof}
Similar to the proof of Theorem \ref{arielle}. 
\end{proof}

Next, we express Theorems \ref{arielle} and \ref{isabelle} in terms of the standard basis for $T_q$. Let $n \ge 1$. For odd $n$, the map $p$ sends

\begin{equation} \label{H1pt5basis} H'_n \mapsto \frac{q^{-n^2}(q^n+q^{-n})(x^ny^n+x^{-n}y^{-n})}{n(q-q^{-1})},\end{equation}

\medskip

\begin{equation} \label{H1basis} H_n \mapsto \frac{q^{-n^2}(q^n+q^{-n})(x^ny^n+x^{-n}y^{-n})}{n(q-q^{-1})}.\end{equation}

For even $n$, $p$ sends

\begin{equation} \label{H2basis} H'_n \mapsto \frac{q^{-n^2}(q^n+q^{-n})(x^ny^n+x^{-n}y^{-n})-2q^n-2q^{-n}}{n(q-q^{-1})},\end{equation}

\medskip

\begin{equation}\label{H3basis} H_n \mapsto \frac{q^{-n^2}(q^n+q^{-n})(x^ny^n+x^{-n}y^{-n})-2q^n-2}{n(q-q^{-1})}.\end{equation}

\section{Proposition \ref{redgoldgreen} revisited}

In Proposition \ref{redgoldgreen}, we listed some identities in $O_q$. In this section we describe what the identities become after we apply the homomorphism $p: O_q \mapsto T_q$. 

Recall the generators $x^{\pm 1}, y^{\pm 1}$ for $T_q$. We will use the following notation. Abbreviate $$z = qyx = q^{-1}xy.$$ Observe that
\begin{equation} \label{scrabble}
z^{-1} = qy^{-1}x^{-1}=q^{-1}x^{-1}y^{-1}, \qquad \qquad zy = q^2yz, \qquad \qquad zx = q^{-2}xz.
\end{equation}

\textcolor{black}{To start, we will express} \eqref{bluebird}, \eqref{thetaprime}, \eqref{theta}, \eqref{H1pt5}, \eqref{H2}, \eqref{H1}, \eqref{H3} in terms of $z$. The map $p$ sends
\begin{align}
B_{1,r} &\mapsto \frac{q^ryz^r + q^ry^{-1}z^{-r}}{q^{1/2}(q-q^{-1})} = \frac{q^{r-1}x^{-1}z^{r+1} + q^{r+1}xz^{-r-1}}{q^{1/2}(q-q^{-1})},&r \in \mathbb{Z},\label{141} \\
\Theta'_n &\mapsto (q-q^{-1})^{-1}\bigl([n+1]_q(z^n+z^{-n}) + \sum_{\ell=1}^{n-1} (q^{n-2\ell}+q^{2\ell-n})z^{n-2\ell}\bigr),&n \ge 1, \label{142a} \\
\Theta_n &\mapsto (q-q^{-1})^{-1}\bigl([n+1]_q(z^n+z^{-n}) + (q^{2-n}+q^{-n})\sum_{\ell=1}^{n-1} z^{n-2\ell}\bigr) \label{143a}, &n \ge 1 ,\\
H'_n &\mapsto \frac{(q^n+q^{-n})(z^n+z^{-n})-2\delta_{n,ev}(q^n+q^{-n})}{n(q-q^{-1})} \label{144}, &n \ge 1, \\
H_n &\mapsto \frac{(q^n+q^{-n})(z^n+z^{-n})-2\delta_{n,ev}(q^n+1)}{n(q-q^{-1})} \label{145}, &n \ge 1.
\end{align}

For later use, we mention a fact about the elements $\{\Theta_n\}_{n=1}^{\infty}$ of $O_q$. Using \eqref{143a}, one readily shows that for $n \ge 3$, the map $p$ sends
\begin{equation}\Theta_n - q^{-2} \Theta_{n-2} \mapsto \frac{[n+1]_q(z^n+z^{-n}) - [n-3]_q (z^{n-2}+z^{2-n})}{q-q^{-1}}. \label{sosbysza} \end{equation}

\textcolor{black}{Next, we use equations \eqref{141}--\eqref{145} to describe what \eqref{hhoq}--\eqref{unlabeled1} look like at the level of $T_q$. The case of \eqref{hhoq} is trivial and will be omitted.}


We now consider what \eqref{tasty_boq} looks like after applying the map $p$.

\begin{proposition} For $m \ge 1$ and $r \in \mathbb{Z}$, the following holds in $T_q$:
\begin{equation}\begin{split} \label{symmetric_b}
&[z^m+z^{-m},q^ryz^r+q^ry^{-1}z^{-r}] \\ &=(q^{m}-q^{-m})(q^{r+m}yz^{r+m}+q^{r+m}y^{-1}z^{-r-m}-q^{r-m}yz^{r-m}-q^{r-m}y^{-1}z^{m-r}).
\end{split}\end{equation}
\end{proposition}
\begin{proof}
Evaluate \eqref{tasty_boq} using \eqref{141}, \eqref{144}, \eqref{145} and routine algebraic manipulation.
\end{proof}
Next, we consider what \eqref{awesomesauceprimeoq} looks like after applying the map $p$. 
\begin{proposition} \label{oc} For $m \ge 1 $ and $r \in \mathbb{Z}$, the following holds in $T_q$:

\vspace{.1in}

\underline{\rm \textbf{Case $m=1$}}
\begin{align}
&(q-q^{-1})^{-1}[z+z^{-1},q^ryz^r+q^ry^{-1}z^{-1}] \nonumber \\ &= q^{r+1}yz^{r+1}+q^{r+1}y^{-1}z^{-r-1}-q^{r-1}yz^{r-1}-q^{r-1}y^{-1}z^{-r+1}. \nonumber
\end{align}

\underline{\rm \textbf{Case $m= 2$}}
\begin{align}
0&= [3]_q [z^2+z^{-2},q^ryz^r +q^r y^{-1}z^{-r}] \nonumber \\
&\quad-[2]_q [z+z^{-1},q^{r+1}yz^{r+1} +q^{r+1} y^{-1}z^{-r-1}]_{q^2} \nonumber \\
&\quad-[2]_q [z+z^{-1},q^{r-1}yz^{r-1} +q^{r-1} y^{-1}z^{-r+1}]_{q^{-2}} \nonumber  
\end{align}

\underline{\rm \textbf{Case $m\ge 3$}}

\begin{align} \label{P142}
0&= [m+1]_q [z^m+z^{-m},q^ryz^r +q^r y^{-1}z^{-r}] \nonumber \\&\quad+ \sum_{\ell = 1}^{m-1} (q^{m-2\ell}+q^{2\ell-m}) [z^{m-2\ell}, q^ryz^r +q^r y^{-1}z^{-r}] \nonumber \\
&\quad+[m-1]_q [z^{m-2}+z^{2-m},q^ryz^r +q^r y^{-1}z^{-r}] \nonumber \\&\quad+ \sum_{\ell = 1}^{m-3} (q^{m-2\ell-2}+q^{2\ell-m+2}) [z^{m-2\ell-2}, q^ryz^r +q^r y^{-1}z^{-r}] \nonumber \\
&\quad-[m]_q [z^{m-1}+z^{1-m},q^{r+1}yz^{r+1} +q^{r+1} y^{-1}z^{-r-1}]_{q^2} \nonumber \\
&\quad- \sum_{\ell = 1}^{m-2} (q^{m-2\ell-1}+q^{2\ell-m+1}) [z^{m-2\ell-1}, q^{r+1}yz^{r+1} +q^{r+1} y^{-1}z^{-r-1}]_{q^2} \nonumber \\
&\quad-[m]_q [z^{m-1}+z^{1-m},q^{r-1}yz^{r-1} +q^{r-1} y^{-1}z^{-r+1}]_{q^{-2}} \nonumber \\
&\quad- \sum_{\ell = 1}^{m-2} (q^{m-2\ell-1}+q^{2\ell-m+1}) [z^{m-2\ell-1}, q^{r-1}yz^{r-1} +q^{r-1} y^{-1}z^{-r+1}]_{q^{-2}}. \nonumber 
\end{align}

\end{proposition}
\begin{proof}
Evaluate \eqref{awesomesauceprimeoq} using \eqref{141} and \eqref{142a}.
\end{proof}

We express Proposition \ref{oc} in a more compact form. In this proposition, we assume $m \ge 2$.

\begin{proposition}
 For $r \in \mathbb{Z}$, the following holds in $T_q$:

\scalebox{0.8}{\parbox{.5\linewidth}{%
\begin{align}
0 &=[m+1]_q [z^m+z^{-m},q^ryz^r +q^r y^{-1}z^{-r}] \nonumber \\ &\quad+[m-1]_q [z^{m-2}+z^{2-m},q^ryz^r +q^r y^{-1}z^{-r}] \nonumber  \\
&\quad-[m]_q [z^{m-1}+z^{1-m},q^{r+1}yz^{r+1} +q^{r+1} y^{-1}z^{-r-1}]_{q^2} \nonumber \\ &\quad-[m]_q [z^{m-1}+z^{1-m},q^{r-1}yz^{r-1} +q^{r-1} y^{-1}z^{-r+1}]_{q^{-2}} \nonumber \\
&\quad- \left[\frac{(q^{m-3}+q^{3-m})(z^{m-1}+z^{1-m})-(q^{m-1}+q^{1-m})(z^{m-3}+z^{3-m})}{(qz-q^{-1}z^{-1})(q^{-1}z-qz^{-1})}, q^{r+1}yz^{r+1} +q^{r+1}y^{-1}z^{-r-1}\right]_{q^2} \nonumber \\
&\quad- \left[\frac{(q^{m-3}+q^{3-m})(z^{m-1}+z^{1-m})-(q^{m-1}+q^{1-m})(z^{m-3}+z^{3-m})}{(qz-q^{-1}z^{-1})(q^{-1}z-qz^{-1})}, q^{r-1}yz^{r-1} +q^{r-1} y^{-1}z^{1-r}\right]_{q^{-2}} \nonumber \\
&\quad+ \left[\frac{(q^{m-2}+q^{2-m})(z^m+z^{-m})-(q^m+q^{-m})(z^{2-m}+z^{m-2})}{(qz-q^{-1}z^{-1})(q^{-1}z-qz^{-1})}, q^ryz^r +q^r y^{-1}z^{-r}\right] \nonumber \\&\quad+\left[\frac{(q^{m-4}+q^{4-m})(z^{m-2}+z^{2-m})-(q^{m-2}+q^{2-m})(z^{4-m}+z^{m-4})}{(qz-q^{-1}z^{-1})(q^{-1}z-qz^{-1})}, q^ryz^r +q^r y^{-1}z^{-r}\right] \nonumber.
\end{align}
}}

\end{proposition}

\begin{proof}
Routine consequence of Proposition \ref{oc}.
\end{proof}

Next, we consider what \eqref{awesomesauceoq} looks like after applying the map $p$.

\begin{proposition} \label{gross} For $r \in \mathbb{Z}$, the following holds in $T_q$:

\vspace{.1in}

\underline{\rm \textbf{Case $m=1$}}

\begin{align}
&(q-q^{-1})^{-1}[z+z^{-1},q^ryz^r+q^ry^{-1}z^{-1}] \nonumber \\ &= q^{r+1}yz^{r+1}+q^{r+1}y^{-1}z^{-r-1}-q^{r-1}yz^{r-1}-q^{r-1}y^{-1}z^{-r+1}. \nonumber
\end{align}

\underline{\rm \textbf{Case $m=2$}}

\begin{align}
0 =\, &[3]_q [z^2+z^{-2},q^ryz^r +q^r y^{-1}z^{-r}] \nonumber\\
&-[2]_q [z+z^{-1},q^{r+1}yz^{r+1} +q^{r+1} y^{-1}z^{-r-1}]_{q^2} \nonumber \\
&-[2]_q [z+z^{-1},q^{r-1}yz^{r-1} +q^{r-1} y^{-1}z^{-r+1}]_{q^{-2}} \nonumber
\end{align}

\underline{\rm \textbf{Case $m\ge 3$}}

\begin{align}
0=\,&[m+1]_q [z^m+z^{-m},q^ryz^r +q^r y^{-1}z^{-r}] \nonumber\\&+(q^{2-m}+q^{-m}) \sum_{\ell = 1}^{m-1} [z^{m-2\ell}, q^ryz^r +q^r y^{-1}z^{-r}] \nonumber \\
&+[m-1]_q [z^{m-2}+z^{2-m},q^ryz^r +q^r y^{-1}z^{-r}] \nonumber\\&+ (q^{4-m}+q^{2-m}) \sum_{\ell = 1}^{m-3} [z^{m-2\ell-2}, q^ryz^r +q^r y^{-1}z^{-r}] \nonumber \\
&-[m]_q [z^{m-1}+z^{1-m},q^{r+1}yz^{r+1} +q^{r+1} y^{-1}z^{-r-1}]_{q^2} \nonumber \\
&- (q^{3-m}+q^{1-m}) \sum_{\ell = 1}^{m-2} [z^{m-2\ell-1}, q^{r+1}yz^{r+1} +q^{r+1} y^{-1}z^{-r-1}]_{q^2} \nonumber \\
&-[m]_q [z^{m-1}+z^{1-m},q^{r-1}yz^{r-1} +q^{r-1} y^{-1}z^{-r+1}]_{q^{-2}} \nonumber \\
&-(q^{3-m}+q^{1-m}) \sum_{\ell = 1}^{m-2}  [z^{m-2\ell-1}, q^{r-1}yz^{r-1} +q^{r-1} y^{-1}z^{-r+1}]_{q^{-2}}. \nonumber
\end{align}
\end{proposition}
\begin{proof}
Evaluate \eqref{awesomesauceoq} using \eqref{141} and \eqref{143a}.
\end{proof}

We express Proposition \ref{gross} in a more compact form. In this proposition, we assume $m \ge 2$.

\begin{proposition}
 For $r \in \mathbb{Z}$, the following holds in $T_q$:
\begin{align}
0=\,&[m+1]_q[z^m+z^{-m},q^ryz^r+q^ry^{-1}z^{-r}]\nonumber\\&+[m-1]_q[z^{m-2}+z^{2-m},q^ryz^r+q^ry^{-1}z^{-r}] \nonumber \\
&- [m]_q[z^{m-1}+z^{1-m},q^{r+1}yz^{r+1}+q^{r+1}y^{-1}z^{-r-1}]_{q^2}\nonumber\\&- [m]_q[z^{m-1}+z^{1-m},q^{r-1}yz^{r-1}+q^{r-1}y^{-1}z^{-r+1}]_{q^{-2}} \nonumber \\
&-(q^{3-m}+q^{1-m})\left[\frac{z^{m-2}-z^{2-m}}{z-z^{-1}}, q^{r+1}yz^{r+1}+q^{r+1}y^{-1}z^{-r-1}\right]_{q^2} \nonumber \\
&- (q^{3-m}+q^{1-m})\left[\frac{z^{m-2}-z^{2-m}}{z-z^{-1}}, q^{r-1}yz^{r-1}+q^{r-1}y^{-1}z^{-r+1}\right]_{q^{-2}} \nonumber \\
&+(q^{3-m}+q^{1-m})\left[ \frac{q(z^{m-3}-z^{3-m})+q^{-1}(z^{m-1}-z^{1-m})}{z-z^{-1}}, q^ryz^r + q^ry^{-1}z^{-r} \right]. \nonumber
\end{align}

\end{proposition}

\begin{proof}
Routine consequence of Proposition \ref{gross}.
\end{proof}

Finally, we consider what \eqref{unlabeled1} looks like after applying $p$.

\begin{proposition}
For $r, s \in \mathbb{Z}$, the following hold in $T_q$:

\vspace{.15in}
\underline{\rm \textbf{Case $s > r+1$}}
\begin{equation}
\begin{split}
&q^{r+s+1}\biggl([yz^r+y^{-1}z^{-r}, yz^{s+1}+y^{-1}z^{-s-1}]_q+[yz^{s}+y^{-1}z^{-s},yz^{r+1}+y^{-1}z^{-r-1}]_q\biggr) \\
&= (q^{s-r+2}-q^{r-s-2})(z^{s-r+1}+z^{r-s-1}) - (q^{s-r-2}-q^{r-s+2}) (z^{s-r-1}+z^{r-s+1}).  \nonumber
\end{split}
\end{equation}

\underline{\rm \textbf{Case $r > s+1$}}
\begin{equation}
\begin{split}
&q^{r+s+1}\biggl([yz^r+y^{-1}z^{-r}, yz^{s+1}+y^{-1}z^{-s-1}]_q+[yz^{s}+y^{-1}z^{-s},yz^{r+1}+y^{-1}z^{-r-1}]_q\biggr) \\
&= (q^{r-s+2}-q^{s-r-2})(z^{r-s+1}+z^{s-r-1}) - (q^{r-s-2}-q^{s+2-r}) (z^{r-s-1}+z^{s-r+1}). \nonumber
\end{split}
\end{equation}

\underline{\rm \textbf{Case $s = r+1$}}
\begin{equation}
\begin{split}
&q^{2r+2}\biggl((q-q^{-1})(yz^{r+1}+y^{-1}z^{-r-1})^2+[yz^{r}+y^{-1}z^{-r},yz^{r+2}+y^{-1}z^{-r-2}]_q\biggr) \\
&=(q^3-q^{-3})(z^2+z^{-2})+2(q-q^{-1}). \nonumber
\end{split}
\end{equation}

\underline{\rm \textbf{Case $r = s+1$}}
\begin{equation}
\begin{split}
&q^{2s+2}\biggl((q-q^{-1})(yz^{s+1}+y^{-1}z^{-s-1})^2+[yz^{s}+y^{-1}z^{-s},yz^{s+2}+y^{-1}z^{-s-2}]_q\biggr) \\
&=(q^3-q^{-3})(z^2+z^{-2})+2(q-q^{-1}).    \nonumber
\end{split}
\end{equation}

\underline{\rm \textbf{Case $s=r$}}
\begin{equation}
q^{2r+1}[yz^r+y^{-1}z^{-r}, yz^{r+1}+y^{-1}z^{-r-1}]_q = (q^2-q^{-2})(z+z^{-1}). \nonumber
\end{equation}
\end{proposition}

\begin{proof}
Follows from \eqref{sosbysza} and routine algebra.
\end{proof}

\section{Acknowledgments}

The author is presently a graduate student at the University of Wisconsin–Madison.
He is grateful to his research advisor, Professor Paul Terwilliger, for many
of the resources used and for insightful comments to improve the clarity and flow
of this paper.

\noindent \textsuperscript{a}Owen Goff \hfil\break
\noindent Department of Mathematics \hfil\break
\noindent University of Wisconsin \hfil\break
\noindent 480 Lincoln Drive \hfil\break
\noindent Madison, WI 53706-1388 \hfil\break
\noindent United States of America \hfill\break
\noindent Email: {\tt ogoff@wisc.edu }\hfil\break
\end{document}